\documentclass[reqno]{amsart}
\usepackage{graphicx}
\usepackage[lite]{amsrefs}
\usepackage{hyperref}
\usepackage{xcolor}
\hypersetup{
    colorlinks,
    linkcolor={red!50!black},
    citecolor={blue!50!black},
    urlcolor={blue!80!black}
}

\usepackage{amssymb}

\usepackage[all,cmtip]{xy}

\usepackage{multicol}

\newcommand{\R}{\mathbb{R}}

\newcommand{\hyp}{\mathbb{H}}

\numberwithin{equation}{section}

\usepackage{enumerate}
\usepackage{caption}
\usepackage{subcaption}
\theoremstyle{plain} 
\newtheorem{thm}[equation]{Theorem}

\newtheorem{prop}[equation]{Proposition}

\theoremstyle{definition}

\theoremstyle{remark}

\newcommand{\cddot}{\mathbin{{\cdot}\,{\cdot}}}

\begin{document}

\title{Infinite $\{3,7\}$-surface in $\hyp^3$}

\author{Dami Lee} \email{damilee@indiana.edu}
\author{Casey Zhao} \email{czhao4@uw.edu}
\address{}




\date{}


\maketitle


 \begin{abstract}
 Objects with large symmetry groups have been an interest for many mathematicians. A classical question in geometry is whether a surface with certain geometric features, such as completeness, curvature, etc..., can embed in $\mathbb{R}^3.$ In \cite{orig}, Lee constructs an infinite $\{3,7\}$-surface in $\R^3$ by gluing together prisms and antiprisms, in an attempt to find a periodic surface in $\R^3$ that is cover of Klein's quartic. While Lee's construction shows that such construction self-intersects in $\R^3$, it does not prove nor disprove the possibility of an embedding. In this paper, we explore a possible embedding of the genus three Klein's quartic, or its cover, in hyperbolic space.\end{abstract}

 \section{Main Result}
In this paper, we mimic the construction of an infinite $\{3,7\}$-surface in $\R^3$ shown in \cite{orig} and build the surface in the Poinc\'are ball, which we denote by $\hyp^3$. We observe the following result.

  \begin{thm} A hyperbolic polyhedral surface $\{3,7\}$ in $\hyp^3$ has no self-intersection if its side length is greater than some $\ell > 0.48.$ \end{thm}
 
 This lower bound is achieved via experiment. 
 \begin{figure}[h]
     \centering
     \includegraphics[width=0.63\linewidth]{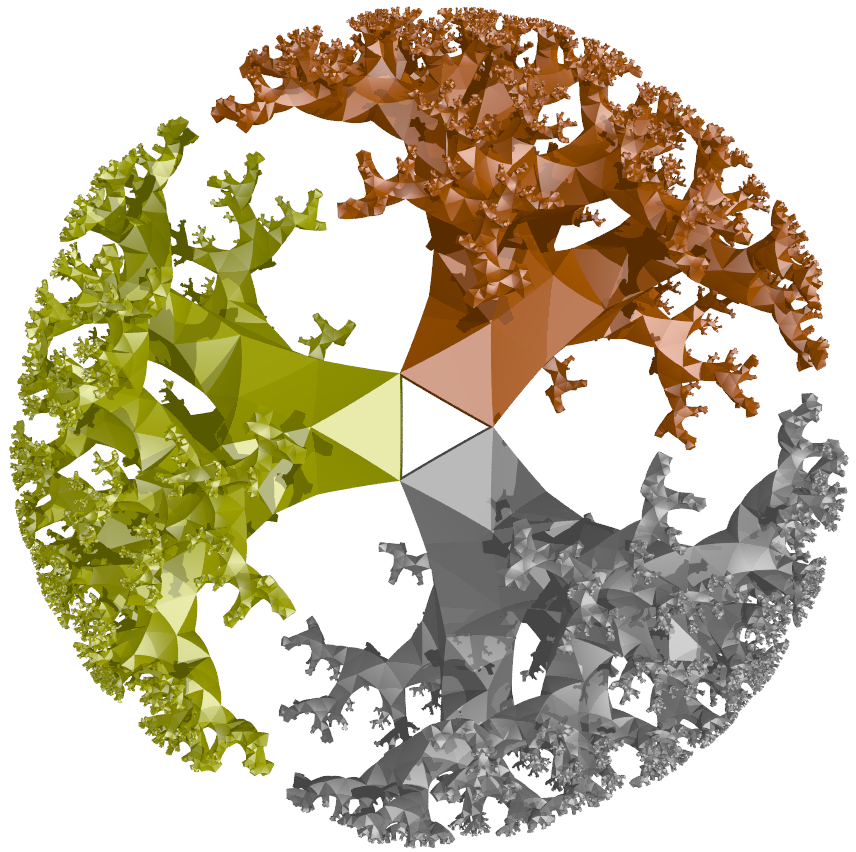}
     \caption{Surface with side length $s=0.53$}
     \label{fig:big}
 \end{figure}
 
The outline of this paper is as follows: \begin{itemize}
    \item In section~\ref{sec:intro}, we give a brief summary of the Euclidean $\{3,7\}$-surface from \cite{orig}.
    \item  In section~\ref{sec:methods}, we outline the algorithm for our experiment that produced our main result.
    \item In section~\ref{sec:results}, we display figures with various parameters.
 \end{itemize}
 
The authors would like to thank WXML (Washington eXperimental Mathematics Lab) organizers for the opportunity to run this project. The authors would also like to give credit to Jacob Ogden for helpful insights in proving Proposition~\ref{prop3.1} .    
 
 \section{Introduction}
 \label{sec:intro}

 In this section we summarize the work of \cite{orig}.

The Schl{\"a}fli symbol $\{p,q\}$ implies that a polyhedral surface is tiled by regular $p$-gons and is $q$-valent at every vertex. The key idea in \cite{orig} is to view a polyhedral surface as the boundary of a polyhedron obtained by gluing polyhedral solids. The solids used here are uniform prisms over triangles and uniform antiprisms over squares. By uniform, we mean right and equilateral. To obtain a boundary surface tiled exclusively by triangles, we glue prisms and antiprisms along their square sides. Start with a triangular prism, and glue the base of an antiprism to each of the three square sides of the prism. At this step, there are two different ways we can glue a prism to each square face of the antiprisms:
 \begin{figure}[h]
     \centering
     \includegraphics[width=0.7\linewidth]{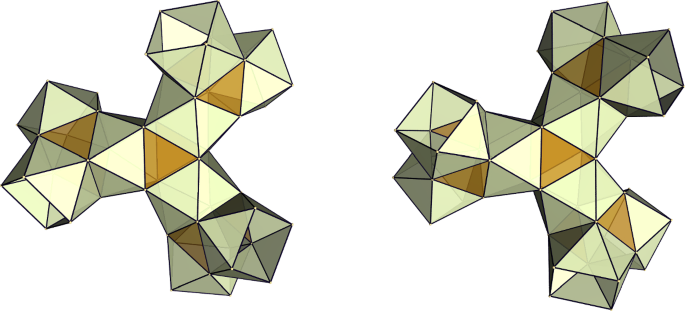}
     \caption{Twists in opposite directions, reprinted from \cite{orig}}
     \label{fig:twists}
 \end{figure}

We choose a direction of ``twisting'' (Figure~\ref{fig:twists}) and stick with this choice throughout the entire construction of the surface\footnote{In \cite{orig} Lee does not consider other gluing patterns of polyhedra where twist directions vary between steps.}. After just a few iterations, the surface self-intersects, hence can only be immersed in $\R^3$ but not embedded.
 
If a surface were to embed in $\R^3$ or furthermore be periodic in $\R^3,$ i.e., invariant under a lattice of translations, then one can consider the quotient of the periodic surface. Periodic surfaces have been studied in various fields, mathematically in Coxeter's work in regular polyhedra \cite{coxeter1937regular}, in the theory of infinite minimal surfaces \cite{schoen1970infinite}, and also experimentally in crystallography and polymer chemistry \cite{sheng}.

 In this paper, we replace the Euclidean polyhedral solids in \cite{orig} with hyperbolic polyhedral solids and use the same gluing pattern of polyhedra to construct the surface in $\hyp^3$ in an attempt to ``fix'' the self-intersecting property of the surface in $\R^3$. If the side length of all polyhedra is small, the resulting hyperbolic surface resembles the Euclidean surface and hence still self-intersects. However, we observe via experiment that if the side length is larger than some minimum length $\ell$, the surface does not self-intersect.

 \section{Methods}
 \label{sec:methods}
 To create our visualizations, we use a combination of Python and POV-Ray. We first calculate the Euclidean coordinates of each of the surface's vertices using Python. The data is then rewritten as code readable by the ray-tracing program, POV-Ray, which outputs the drawing of our surface.
 
 We start by finding the vertices of a uniform triangular prism and a uniform square antiprism with a given hyperbolic side length, such that their centers of mass are located at the origin of the Poincar\'e ball. It is easy to calculate the coordinates of their vertices in this location, since
 
 \begin{prop}
    The vertices of a convex uniform Euclidean polyhedron contained in, and centered at the origin of the Poinc\'are ball are also the vertices of a convex uniform hyperbolic polyhedron centered at the origin.\label{prop3.1}
 \end{prop}

 \begin{proof}
    We verify that the hyperbolic distances between incident vertices of a convex uniform Euclidean polyhedron centered at the origin $O$ of the Poinc\'are ball are the same. Let $|\cddot|$ denote the Euclidean distance between two points, and $||\cddot||$ the hyperbolic distance. Pick two edges $AB$ and $CD$ from the polyhedron. Since the vertices of a uniform polyhedron lie on a sphere, we have
    \begin{align*}
        |AO|=|BO|&=|CO|=|DO|,\tag{1}\\
        ||AO||=||BO||&=||CO||=||DO||.\tag{2}
    \end{align*}
    
    To prove that $||AB||=||CD||$, it is enough to show that $\angle AOB=\angle COD$. But this follows from $(1)$ and $|AB|=|CD|$, hence the side lengths of the polyhedron are equal in the hyperbolic sense.
\end{proof}

 In particular, Proposition~\ref{prop3.1} holds for a uniform prism or antiprism centered at the origin. Thus, finding the vertices of a hyperbolic prism or antiprism at the origin is only a matter of scaling down a set of vertices of a Euclidean prism or antiprism.

To assemble our surface, we start with a copy of a prism or antiprism centered at the origin, and relocate it to its final location, using the sequence of hyperbolic isometries described in the proposition below. We follow the gluing pattern described in section~\ref{sec:intro}. When gluing a prism to an antiprism, we choose a counterclockwise twist at each step, as in the example on the left of Figure \ref{fig:twists}.
\begin{prop}
 An orientation-preserving isometry in $\hyp^3$ can be written as a composition of rotations about the origin and plane reflections. 
\end{prop}

\begin{proof}
 The idea of the proof is equivalent to the idea that an isometry of $\R^3$ is a composition of an element of $\text{SO}(3)$ and an arbitrary translation. While the proof holds for any object in $\hyp^3,$ for the purpose of this paper, we illustrate our proof specifically with uniform antiprisms over squares.\vfill\hfill\pagebreak
 
 Figure~\ref{fig:iso} outlines how we can move the antiprism labeled $1$, call it $AP1$, so that the square base of $AP1$ facing towards us is mapped to an arbitrary target square $a$. In particular, we want each of the three colored vertices marked on the base of $AP1$ to map to its correspondingly colored vertex on $a$. Specifying these three vertices determines where $AP4$ lies with respect to $a$; in this example, we want $a$ to be the base of $AP4$ facing towards us rather than the one facing away.\medskip 
 
 \begin{figure}
    \centering
    \includegraphics[width=0.75\linewidth]{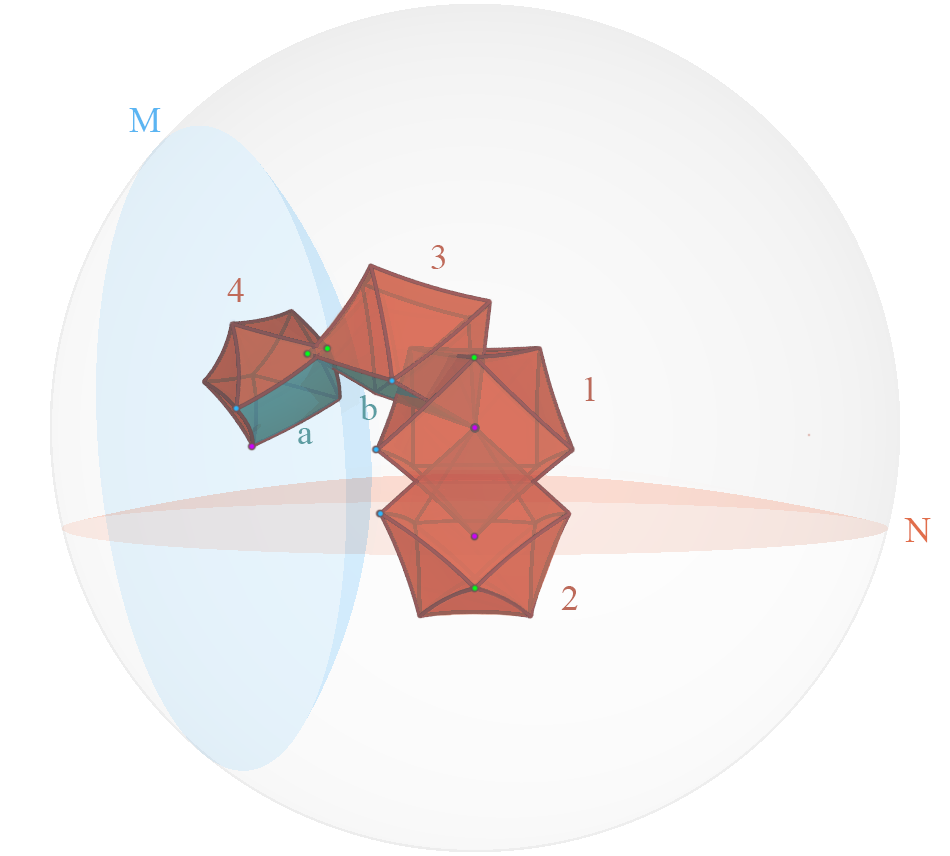}
    \caption{Sequence of isometries}
    \label{fig:iso}
\end{figure}

\noindent Step 1: Calculate the perpendicular bisector $N$ between the purple vertex of $AP1$ and the origin, and reflect $AP1$ over $N$ to get $AP2$. Note how the reflection maps the purple vertex of $AP1$ to the origin.\medskip

\noindent Step 2: Similarly, calculate the perpendicular bisector $M$ between the purple vertex of square $a$ and the origin, and reflect $a$ over $M$ to get square $b$. Again, note how the reflection maps the purple vertex of $a$ to the origin. \medskip

\noindent Step 3: Since the purple vertices of square $b$ and $AP2$ are both located at the origin, we can rotate $AP2$ about the origin so that the blue and green vertices of $AP2$ align with those of the square $b$. This gives us $AP3$.\medskip

\noindent Step 4: Finally, reflect $AP3$ over $M$ to get our final result, $AP4$.\medskip

We can carry out a similar procedure for a triangular prism, or in general, for any other polyhedron with any given initial and final locations.
\end{proof}

After gluing down a prism or antiprism, we collect the vertices $p,q,r$ of each new triangular face. Then, we find the hyperbolic plane passing through $p,q$ and $r$, and ``cut'' $\triangle pqr$ from the plane, using POV-Ray's CSG intersection.

 \section{Results}
 \label{sec:results}
Figure~\ref{fig:1} shows the ``top'' and ``side'' views of the surface for varying side lengths $s$ after 11 iterations\footnote{Due to time and resource limitations, we were only able to iterate the gluing procedure eleven times, where a single iteration consists of gluing a prism to each open square face, and then two antiprisms to each prism. Rendering the animations on our website \cite{web} was computationally intensive, so in those we opted for flat face approximations instead of spherical ones.}:
 
 \begin{figure}[h]
	\centering
	\begin{subfigure}[t]{0.35\linewidth}
		\centering
		\includegraphics[width=\linewidth]{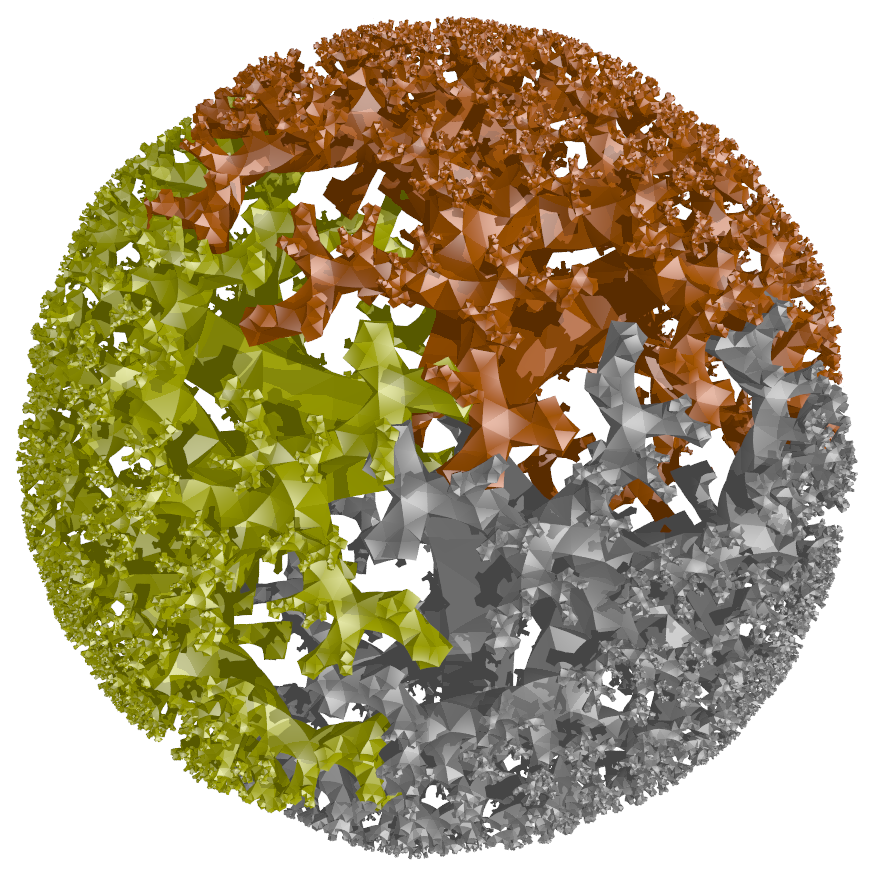}
		\caption{ $s=0.48$, top view}\label{fig:a}		
	\end{subfigure}
	\quad
	\begin{subfigure}[t]{0.35\linewidth}
		\centering
		\includegraphics[width=\linewidth]{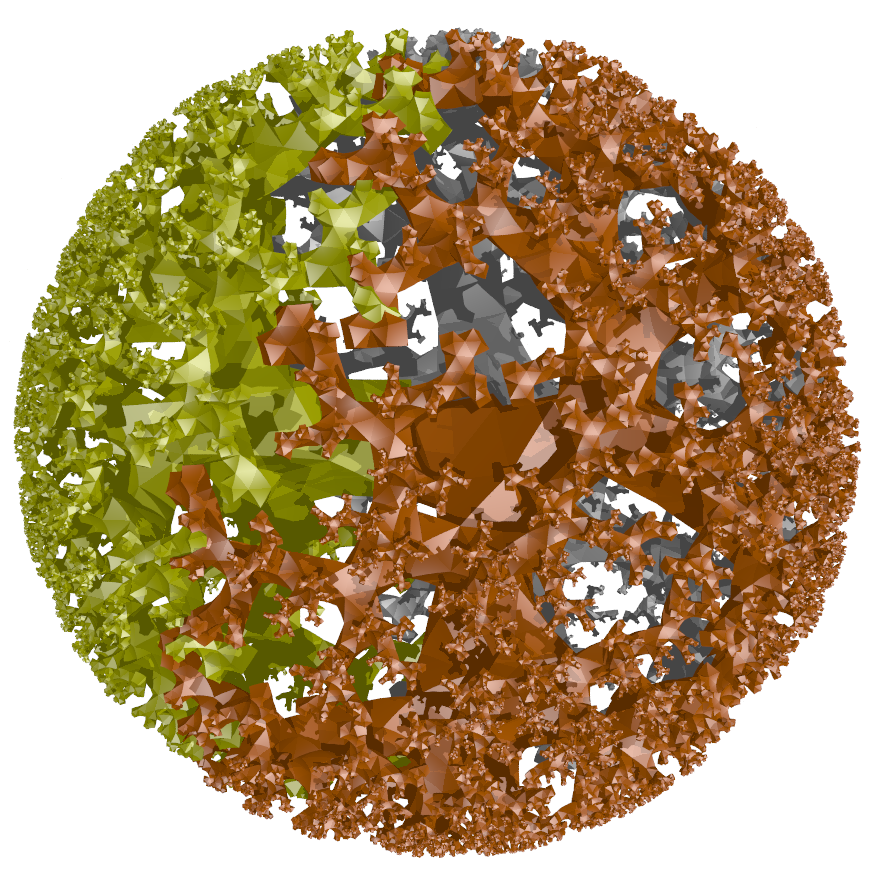}
		\caption{ $s=0.48$, side view}\label{fig:b}	
	\end{subfigure}\\
	\medskip
	\begin{subfigure}[t]{0.35\linewidth}
		\centering
		\includegraphics[width=\linewidth]{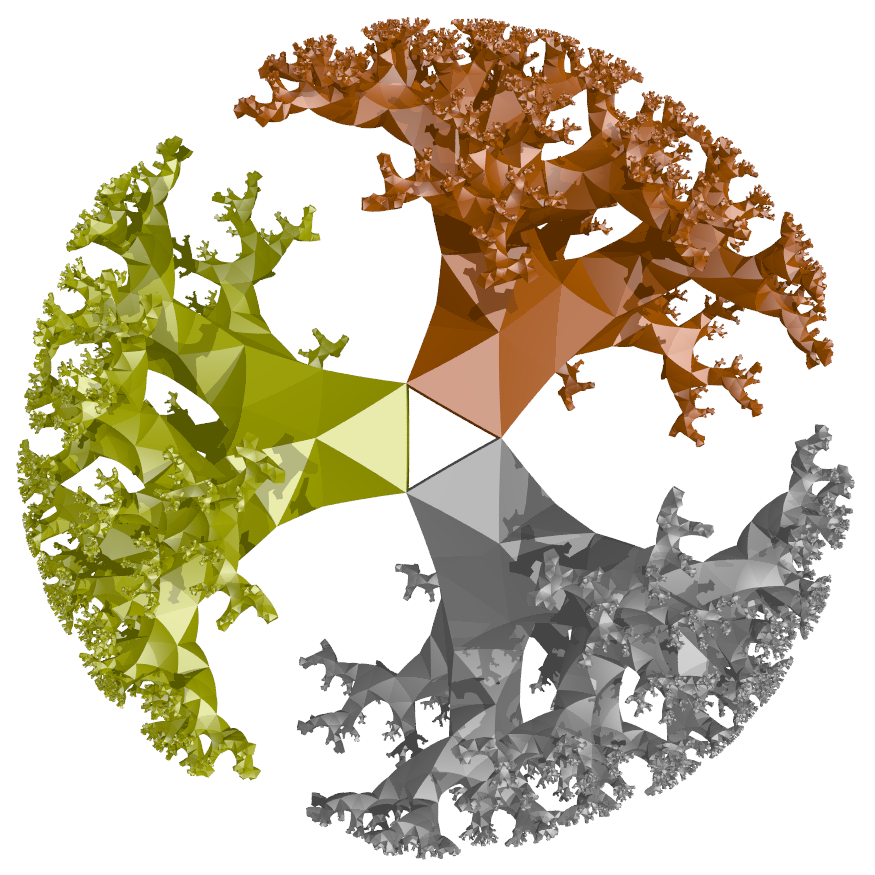}
		\caption{ $s=0.54$, top view}\label{fig:c}		
	\end{subfigure}
	\quad
	\begin{subfigure}[t]{0.35\linewidth}
		\centering
		\includegraphics[width=\linewidth]{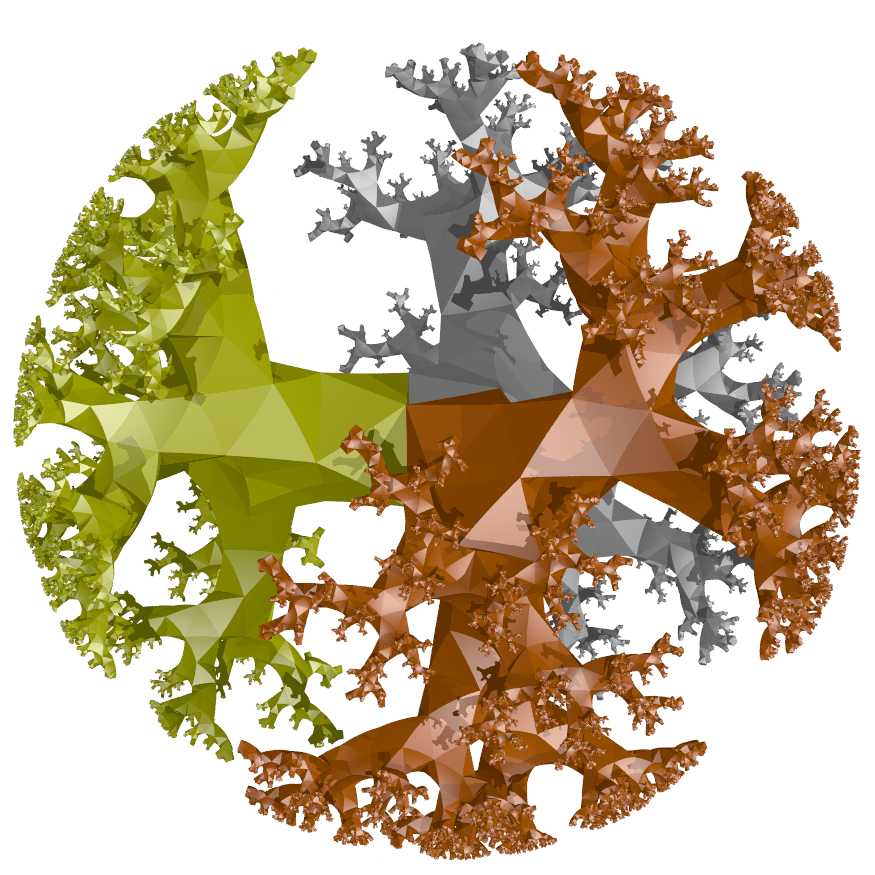}
		\caption{ $s=0.54$, side view}\label{fig:d}	
	\end{subfigure}\\
	\medskip
	\begin{subfigure}[t]{0.35\linewidth}
		\centering
		\includegraphics[width=\linewidth]{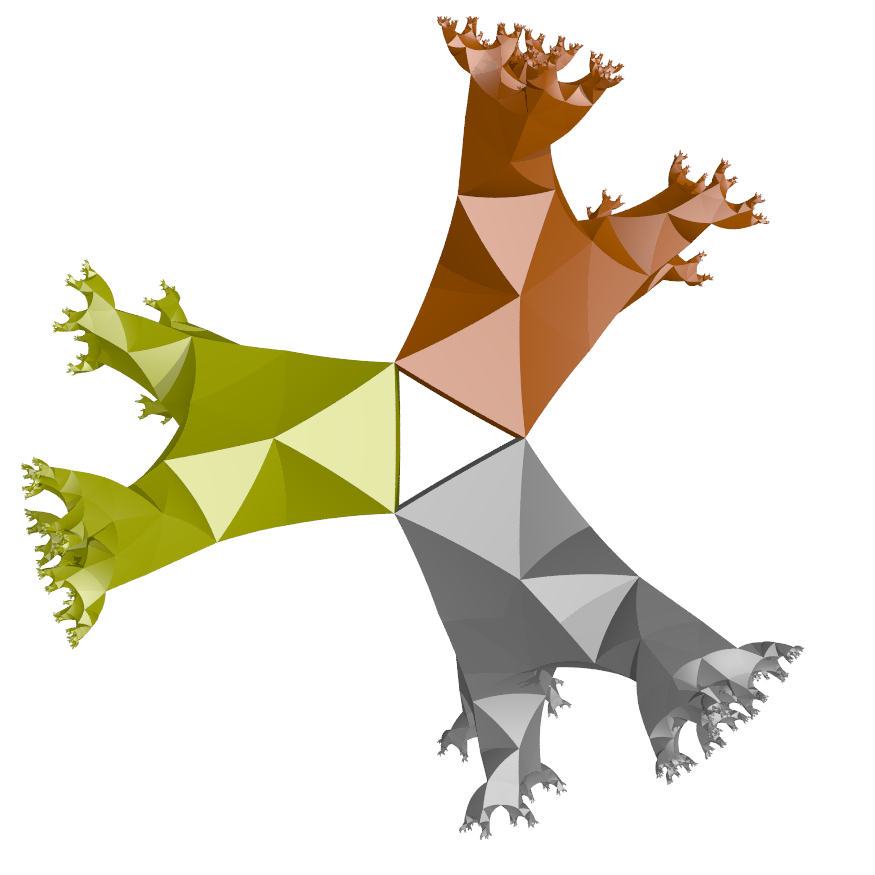}
		\caption{ $s=0.75$, top view}\label{fig:e}		
	\end{subfigure}
	\quad
	\begin{subfigure}[t]{0.35\linewidth}
		\centering
		\includegraphics[width=\linewidth]{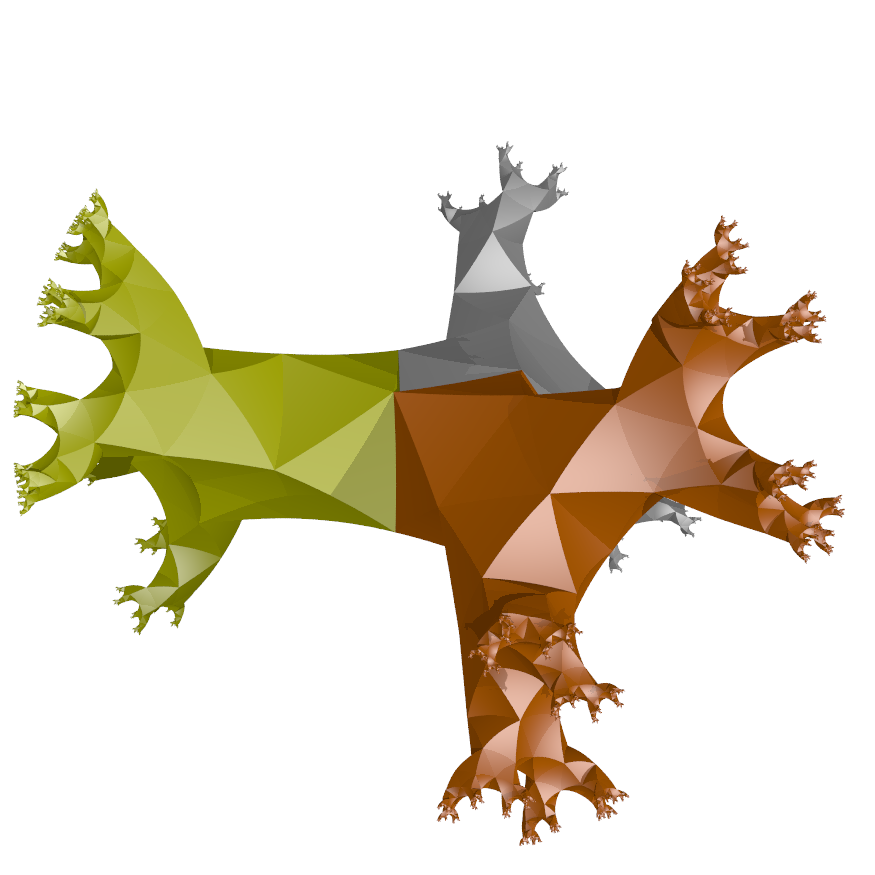}
		\caption{ $s=0.75$, side view}\label{fig:f}	
	\end{subfigure}
	\caption{Surfaces with varying side length.}\label{fig:1}
\end{figure}
See \cite{web} for full-size images and animations, and \cite{code} for the code.


\end{document}